\documentclass[a4paper,reqno]{amsart}
\usepackage[T1]{fontenc}
\usepackage[utf8x]{inputenc}
\usepackage[english]{babel}
\usepackage{amsmath,amsfonts,amssymb,upgreek,comment,mathtools}
\usepackage[11pt,curve,matrix,arrow,frame,graph]{xy}
\usepackage{graphicx}
\usepackage{hyperref}
\usepackage{enumerate}
\usepackage{xcolor}
\usepackage{esint}

\usepackage{amsthm}

\mathtoolsset{showonlyrefs} 

\theoremstyle{plain}
\newtheorem*{theorem*}{Theorem}
\newcounter{intro}

\newtheorem{thm}[intro]{Theorem}
\newtheorem{crl}[intro]{Corollary}

\newtheorem{theorem}{Theorem}[section]
\theoremstyle{definition}
\newtheorem{defi}[theorem]{Definition}
\newtheorem{lemm}[theorem]{Lemma}
\newtheorem{coro}[theorem]{Corollary}
\newtheorem{prop}[theorem]{Proposition}
\newtheorem{theo}[theorem]{Theorem}

\theoremstyle{definition}
\newtheorem{rem}[theorem]{Remark}
\newtheorem{rems}[theorem]{Remarks}

\newcommand{\R}{\ensuremath{\mathbb R}}
\newcommand{\N}{\ensuremath{\mathbb N}}

\newcommand{\bB}{\ensuremath{\mathbb B}}

\newcommand{\bV}{\ensuremath{\mathbb V}}
\newcommand{\cC}{\mathcal{C}}

\newcommand{\cE}{\mathcal{E}}

\newcommand{\cV}{\mathcal{V}}

\DeclareMathOperator{\un}{\mathbf{1}}
\newcommand{\eps}{\ensuremath{\varepsilon}}

\newcommand{\Kato}{\ensuremath{\mbox{k}}}
\newcommand{\Ricm}{\ensuremath{\mbox{Ric}_{\mbox{\tiny{--}}}}}
\newcommand{\katoT}{\ensuremath{\mbox{k}_T(M^n,g)}}

\newcommand{\kato}{\mathrm{k}}
\newcommand{\katot}{\ensuremath{\mbox{k}_t(M^n,g)}}

\DeclareMathOperator{\bg}{ \overline{\emph{g}}}

\DeclareMathOperator{\vol}{\it \nu}

\DeclareMathOperator{\RCD}{RCD}
\DeclareMathOperator{\BE}{BE}
\newcommand{\Ric}{\ensuremath{\mbox{Ric}}}

\DeclareMathOperator{\tr}{tr}

\newcommand{\di}{\mathop{}\!\mathrm{d}}

\newcommand{\measrestr}{%
  \,\raisebox{-.127ex}{\reflectbox{\rotatebox[origin=br]{-90}{$\lnot$}}}\,%
}

\newcommand{\dist}{\mathsf{d}}

\DeclareMathOperator{\CD}{CD}

\def\cH{\mathcal H}
\def\mcK{\mathcal K}

\def\ra{\rangle}
\def\la{\langle}

\newcommand{\df}{\coloneqq}

\newcommand{\odist}{\overline{\dist}}
\newcommand{\omu}{\overline{\mu}}

\newcommand{\cref}[1]{Corollary~\ref{#1}}
\newcommand{\dref}[1]{Definition~\ref{#1}}

\newcommand{\lref}[1]{Lemma~\ref{#1}}
\newcommand{\pref}[1]{Proposition~\ref{#1}}

\newcommand{\tref}[1]{Theorem~\ref{#1}}

\title[Kato meets Bakry-\'Emery]{Kato meets Bakry-\'Emery}

\author{Gilles Carron}
\address{G. Carron, Nantes Université, CNRS, Laboratoire de Mathématiques Jean Leray, LMJL, UMR 6629, F-44000 Nantes, France.} 
\email{Gilles.Carron@univ-nantes.fr}

\author{Ilaria Mondello}
\address{I. Mondello, Université Paris Est Cr\'eteil, Laboratoire d'Analyse et Math\'ematiques appliqu\'es, UMR CNRS 8050, F-94010 Créteil, France.}
\email{ilaria.mondello@u-pec.fr}

\author{David Tewodrose}
\address{D. Tewodrose, Vrije Universiteit Brussel, Department of Mathematics and Data Science, Pleinlaan 2, B-1050 Brussel, Belgium.}
\email{david.tewodrose@vub.be}
\date{}

\begin{document}
\maketitle

\begin{abstract}
We prove that any complete Riemannian manifold with negative part of the Ricci curvature in a suitable Dynkin class is bi-Lipschitz equivalent to a finite-dimensional $\RCD$ space, by building upon the transformation rule of the Bakry-\'Emery condition under time change. We apply this result to 
show that our previous results on the limits of closed Riemannian manifolds satisfying a uniform Kato bound \cite{CMT1,CMT2} carry over to limits of complete manifolds. We also obtain a weak version of the Bishop-Gromov monotonicity formula for manifolds satisfying a strong Kato bound.
\end{abstract}

\hfill


\hfill


\section{Introduction}

In a recent series of articles,  we studied the structure of Gromov--Hausdorff limits of closed Riemannian manifolds with Ricci curvature satisfying some uniform Kato type condition \cite{CMT1,CMT2}.  The aim of this paper is to lift technical restrictions, like  the closedness of the approximating manifolds, and to improve our previous results.

For a complete Riemannian manifold $(M^n,g)$ of dimension $n \ge 2$,  define
\begin{equation*}
 \katot := \sup_{x \in M}\int_0^t\int_M H(s,x,y)\Ricm(y) \di \nu_g(y) \di s
\end{equation*}
for any $t>0$, where $H$ is the heat kernel of $(M,g)$,  $\nu_g$ is the Riemannian volume measure and $\Ricm : M \to \R_+$ is the lowest non-negative function such that
\[
\Ric_x  \ge - \Ricm(x)g_x
\]
for any $x \in M$. From our previous work \cite[Corollary 2.5 and Theorem 4.11]{CMT1}, a classical contradiction argument shows that for any $\eps>0$ there exists $\delta>0$ depending on $n$ and $\eps$ only such that if $(M^n,g)$ is closed and satisfies
$$\katoT \le \delta$$ for some $T>0$,
then for any $p\in M$ there exists a pointed $\RCD(0,n)$ space $(X,\dist,\mu,x)$ such that 
$$\dist_{GH}\left( B^M_{\sqrt{T}}(p), B^X_{\sqrt{T}}(x)\right)\le \eps \,\sqrt{T}.$$
Here $\dist_{GH}$ stands for the Gromov--Hausdorff distance. We briefly recall that for $K\in \R$ and $N\in [1,+\infty]$, an $\RCD(K,N)$ space is a metric measure space with a synthetic notion of Ricci curvature bounded below by $K$ and dimension bounded above by $N$. The main result of this paper is a quantitative improvement of the previous fact:

\begin{thm}\label{Maintheo} Let $(M^n,g)$ be a complete Riemannian manifold of dimension $n\geq 2$. Assume that there exist $T>0$ and $\gamma \in (0,1/(n-2))$ such that
\begin{equation}\label{eq:NewDynkin}\tag{D}
\kato_T(M^n,g)\le \gamma.
\end{equation}
Then there exist constants $K\ge 0$ and $N>n$, both depending on $n$ and $\gamma$ only,  and $h\in \cC^2(M)$ with $0 \le h \le C=C(n,\gamma)$, such that the weighted Riemannian manifold $(M,  e^{2h}g , e^{2h} \nu_g )$ satisfies the $\mathrm{RCD}(-K/T,N)$ condition.  Moreover, if 
\begin{equation}\label{eq:NewDynkin2}\tag{D'}
\kato_T(M^n,g)< \frac{1}{3(n-2)}
\end{equation} then we can choose $K=4\kato_T(M^n,g)$,  $N=n+4(n-2)^2\kato_T(M^n,g)$ and $C = 4\katoT$.
\end{thm}

In dimension 2, the previous result holds without restriction on the bound $\gamma$, and it provides a metric conformal and bi-Lipschitz to $g$ with curvature bounded below: this is a simple consequence of our Corollary \ref{coro:jaugeKato} and of the conformal transformation law for the Gauss curvature (see Remark \ref{rem:dim2}).

As a consequence of Theorem \ref{Maintheo}, we establish the existence of cut-off functions with controlled gradient and Laplacian, see Proposition \ref{prop:cutoff}. This implies that all the results proved in \cite{CMT1,CMT2} on closed Riemannian manifolds extend to complete ones, see Section \ref{consman}. 

Theorem \ref{Maintheo} also yields the following result for limit spaces.

\begin{crl}\label{crl:limit}
Let $(X,\dist,\mu,o)$ be the pointed measured Gromov--Hausdorff limit of a sequence of pointed complete weighted Riemannian manifolds $\{(M_\ell^n,g_\ell,c_\ell\nu_{g_\ell,}o_\ell)\}$ where $\{c_\ell\}\subset (0,+\infty)$.  Assume that there exist $T>0$ and $\gamma \in (0,1/(n-2))$ such that \begin{equation}\label{eq:uniform}\tag{UD}
\sup_\ell \kato_T(M_\ell^n,g_\ell)\le \gamma
\end{equation} holds. Then there exist a distance $\odist$ and a measure $\omu$ on $X$ such that $\dist \le \odist \le  C(n,\gamma)\, \dist$, $\mu \le \omu \le C(n,\gamma) \, \mu$, and the space  $(X,\odist,\omu)$ is $\mathrm{RCD}(-K/T,N)$ with $K$ and $N$ given by Theorem \ref{Maintheo}.
\end{crl}

It is worth pointing out that \eqref{eq:uniform} implies the existence of limit points in the pointed measured Gromov--Hausdorff topology: this is a consequence of the local doubling condition obtained in Proposition \ref{prop:DP} and Gromov's compactness theorem \cite[Proposition 5.2]{Gromov}.  Moreover,  by Theorem \ref{Maintheo},  for any $\ell$ there exist a Riemannian metric $\bar{g}_\ell$ and a Borel measure $\bar{\mu}_\ell$ on $M_\ell$ such that $(M_\ell,\bar{g}_\ell,\bar{\mu}_\ell)$ is an $\RCD(-K/T,N)$ space bi-Lipschitz equivalent to  $(M_\ell,g_\ell)$, with bi-Lipschitz bounds independent of $\ell$. As well-known, the class of $\RCD(-K/T,N)$ metric measure spaces is compact for the pointed measured Gromov--Hausdorff topology (by compactness of the larger $\CD(-K/T,N)$ class \cite[Theorem 29.25]{Villani_2009} and stability of the $\RCD(-K/T,\infty)$ condition \cite[Theorem 7.2]{GMS}). Thus $(M_\ell,\bar{g}_\ell,\bar{\mu}_\ell)$ subconverges to an $\RCD(-K/T,N)$ space bi-Lipschitz equivalent to $(X,\dist,\mu)$.

Note that when $n=2$,  the space $(X,\odist)$ is an Alexandrov space with curvature bounded below.

The previous corollary puts us in a position to apply well-known results of the $\RCD$ theory \cite{DPMR, KellMondino,MondinoNaber,Bru__2020,GigliPas} to conclude that $X$ is a rectifiable metric measure space with a constant essential dimension,  see Propositions \ref{prop:Dynlinreg} and \ref{prop:Katoreg}.  This is a significant improvement over the rectifiability result that we proved in \cite[Theorem 4.4]{CMT2}. Corollary \ref{crl:limit} also yields that the singularities of $X$ are no more complicated than those of the boundary elements of the class of smooth $\RCD(K/T,N)$ spaces.

Another important consequence of Theorem \ref{Maintheo} is an almost monotonicity formula for the volume ratio
\[
\cV(x,r) \df \frac{\nu_g(B_{r}(x))}{\omega_n r^n} \, \cdot
\]
Here $\omega_n$ is the Lebesgue measure of the unit Euclidean ball in $\mathbb{R}^n$.  To state that formula,  let us consider a non-decreasing function  $f : (0,T] \to \R_+$ such that
\begin{equation}\label{eq:newstrong}\tag{SK}
f(T) \le \frac{1}{3(n-2)} \qquad \text{and} \qquad \int_0^T \frac{f(s)}{s} \di s <\infty.
\end{equation}
The second bound should be understood as a control on how fast $f$ converges to $0$ in $0$, as it implies that $\displaystyle \lim_{t\to 0+} f(t)=0$.  For any $\tau \in(0, T]$ we set \[ \Phi(\tau) \df  \int_0^{\sqrt{\tau}} \frac{f(s)}{s} \di s.\]

\begin{thm}\label{th:AlmostBishopGromov} Let $(M^n,g)$ be a complete Riemannian manifold such that for any $t \in (0,T]$,
\begin{equation}\label{eq:bound}\tag{K}
\kato_t(M^n,g)\le f(t).
\end{equation} Then for any $x\in M$,  $R \in (0, \sqrt{T}]$,  $\eta \in (0,1-1/\sqrt{2})$ and $r \le (1-\eta)R$,
\[
\cV(x,R)  \exp \left(-\frac{C(n)\Phi(R)}{\eta} \right) \le \cV(x,r) \exp\left(-\frac{C(n)\Phi(r)}{\eta}\right).
\]
\end{thm}

As a corollary,  we obtain the following Hölder regularity result. We denote  by $\mathcal{H}^n$ the $n$-dimensional Hausdorff measure of a metric space. 

\begin{crl}\label{th:reifenberg} Let $(X,\dist,o)$ be the pointed Gromov--Hausdorff limit of a non-collapsed sequence of pointed complete Riemannian manifolds $\{(M_i^n,g_i,o_i)\}$ satisfying \eqref{eq:bound}.  Then the volume density
\[
\uptheta_X(x)\df\lim_{r\to 0+} \frac{\cH^n\left(B_r(x)\right)}{\omega_n r^n}
\]
is well-defined at any $x \in X$, and for any $\alpha\in (0,1)$ there exists $\delta=\delta(n,\alpha,f)>0$ such that the set $\{ x \in X \, : \,  \uptheta_X(x) \ge 1- \delta \}$ is contained in an open $\cC^{\alpha}$ manifold.
\end{crl}

Here the sequence $\{(M_i^n,g_i,o_i)\}$ is non-collapsed if the numbers $\{\nu_{g_i}(B(o_i,\sqrt{T}))\}$ admit a positive lower bound.  Thanks to \eqref{eq:bound},  this sequence satisfies \eqref{eq:uniform} for a possibly smaller value of $T$, because $f(t) \to 0$ as $t \to 0$, thus it necessarily admits pointed measured Gromov--Hausdorff limit points.

Compare to \cite[Corollary 5.20]{CMT2} and note that the integral condition in  \eqref{eq:newstrong} is weaker than the strong Kato condition considered in that paper: see Remark \ref{rem:newstrong} for the details. \\

To establish Theorem \ref{Maintheo}, we adapt the proof of a classical result on Schrödinger operators to show that assumption \eqref{eq:NewDynkin} ensures the existence of  a suitable gauge function $\varphi\in \cC^2(M)$ that satisfies
\[
\Delta_g \varphi-\lambda\Ricm  \varphi\ge -2\beta \varphi
\]
for carefully chosen parameters $\lambda,\beta>0$ depending on $n$ and $\gamma$ only.  We then set
\[
h \df \frac{1}{\lambda} \log \varphi.
\]
With this choice of conformal factor,  the transformation rule under time change (\cite{SturmJFA2018,SturmGAFA2020,HanSturm}, see also Lemma \ref{eq:BEHS}) yields that $(M,  e^{2h}g , e^{2h} \nu_g )$ satisfies a suitable version of the Bochner inequality, namely the Bakry-Émery condition $\mathrm{BE}(K/T,N)$, for $K$ and $N$ depending on $n$ and $\gamma$ only. The conclusion follows since $\mathrm{BE}(K/T,N)$ is equivalent to $\RCD(K/T,N)$ in the setting of weighted Riemannian manifolds, see e.g.~\cite[Theorem 4.9]{sturm2006I} or \cite[Theorem 0.12]{LottVillani}.  The idea of using this transformation was inspired by \cite{CMR} where such a conformal change was made to prove a rigidity result for minimal hypersurfaces in $\R^4$ (see \cite{ChodoshLi} for another proof of this result). This idea was implicity present already in \cite{Elbert_2007}.

The paper is organised as follows.  In Section \ref{TC}, we recall the Bakry-\'Emery condition, time changes, and the aforementioned transformation rule.  Section \ref{pfA} is devoted to proving Theorem \ref{Maintheo}. We give consequences of this theorem for complete Riemannian manifolds in Section \ref{consman} and for limit spaces in Section \ref{conslim}.

\hfill

\noindent \textbf{Acknowledgments:}
The authors are partially supported by the ANR grant ANR-17-CE40-0034: CCEM. The first author is also partially supported by the ANR grant ANR-18-CE40-0012: RAGE.  The third author has been supported by Laboratoire de Mathématiques Jean Leray via the project Centre Henri Lebesgue ANR-11-LABX-0020-01,  by Fédération de Recherche Mathématiques de Pays de Loire via the project Ambition Lebesgue Loire,  and by the Research Foundation – Flanders (FWO) via the Odysseus II programme no.~G0DBZ23N. The authors are also grateful to Mathias Braun, Chiara Rigoni and Christian Rose for helpful comments on a preliminary version of the paper. Lastly, the authors express their gratitude to the anonymous reviewer for constructive suggestions.

\section{The Bakry-\'Emery condition under time change}\label{TC}

Let $(M^n,g)$ be a complete Riemannian manifold of dimension $n \ge 2$.   We write $\Delta_g$ for the non-negative Laplace--Beltrami operator of $(M^n,g)$ defined by
\begin{equation}\label{eq:Green}
 \int_M \langle d\varphi,d\phi\rangle_g\,\di \nu_g=\int_M\varphi\Delta_g\phi\,\di \nu_g
 \end{equation}
for all $\varphi,\phi\in \cC^\infty_c(M)$. We will also use $\Delta_g$ to denote the unique self-adjoint extension of the Laplace--Beltrami operator, which maps $\mathcal{C}_0^\infty(M)$ to $L^2(M, \nu_g)$.  The heat kernel $H$ of $(M^n,g)$ is the kernel of its heat semigroup $(e^{-t\Delta_g})_{t>0}$; in particular,  for any $\phi\in \cC^\infty_c(M)$ and $x \in M$, 
 $$\left(e^{-t\Delta_g}\phi\right)(x)=\int_M H(t,x,y)\phi(y)\di \nu_g(y).$$

We will say that $(M^n,g, \bar\nu)$ is a weighted Riemannian manifold if $\bar\nu$ is a measure absolutely continuous with respect to $\nu_g$ with positive $\mathcal{C}^2$ Radon--Nikodym density. Such a space admits a weighted Laplacian $L$ defined through the Green formula obtained upon replacing $\Delta_g$ by $L$ and $\nu_g$ by $\bar\nu$ in \eqref{eq:Green}.

For any Borel set $A\subset M$, we will write $\un_A$ for the characteristic function of $A$ and $\un$ for the constant function equal to $1$, that is, $\un=\un_M$.  We will denote the spectrum by $\mathrm{spec}$.

If $\mu,\bar{\mu}$ are two Borel measures  on a metric space and $C>0$, we shall write
\[
\mu \le C \bar{\mu}
\]
to denote that $\mu$ is absolutely continuous with respect to $\bar \mu$ with Radon--Nikodym derivative lower than or equal to $C$ $\bar{\mu}$-almost everywhere.

We will write $\mathbb{B}^n_r$ for the Euclidean ball of radius $r$ centered at the origin of $\mathbb{R}^n$.  Lastly,  we will use $C(a_1,\ldots,a_\ell)$ to denote a generic constant depending solely on parameters $a_1,\ldots,a_\ell$ and whose value may change from one line to another.

\subsection{The Bakry-\'Emery condition}\label{sub:BakryEmery}
The Bochner formula for $(M^n,g)$ states that for all $u\in \cC^\infty(M)$,
\begin{equation}\label{eq:Boch}
\la d\Delta_gu,du\ra_g-\frac12 \Delta_g |du|^2_g=\left|\nabla^g du\right|_g^2 +\Ric(du,du).\end{equation}
Introducing the function $\Ricm : M \to \mathbb{R}_+$ defined by
\begin{equation}
\Ricm(x)=\begin{cases}
0 & \text{if } \Ric_x \ge 0,\\
-\min\,{\rm spec}\, (\Ric_x) & \text{otherwise},
\end{cases}
\end{equation}
this yields the so-called Bochner inequality:
\begin{equation}\label{eq:BEg}
\la d\Delta_gu,du\ra_g-\frac12 \Delta_g |du|^2_g\ge \frac{\left(\Delta_g u\right)^2}{n}-\Ricm |du|^2_g . \end{equation}
The Bakry-\'Emery condition is the analogue of \eqref{eq:BEg} for weighted Riemannian manifolds.

\begin{defi}\label{def:BEL} For $K \in \mathbb{R}$ and $N \in [1,+\infty]$, we say that a weighted Riemannian manifold $(M^n,g,\bar\nu)$ with associated weighted Laplacian $L$ satisfies the Bakry-\'Emery condition $\BE(K,N)$ if for any $u\in \cC^\infty(M)$,
\begin{equation}\label{eq:BEL}
\la dLu,du\ra_{g}-\frac12 L \left(|du|^2_{g}\right)\ge \frac{\left(L u\right)^2}{N}+K\,|du|^2_{g}.\end{equation}
\end{defi}

Introduced in the setting of Dirichlet forms in \cite{BakryEmery}, this condition was the first milestone towards the definition and the study of metric measure spaces with a synthetic notion of Ricci curvature bounded from below by $K$ and dimension bounded above by $N$. 

\subsection{Time changes}
We refer to \cite{ChenFukushima} for a nice introduction to time changes in the general setting of symmetric Markov processes. Here we focus on the case of the Brownian motion on $(M^n,g)$ where a time change is obtained by setting $\bar{g}\df e^{2h}g$ and $\bar{\nu} \df e^{2h} \nu_g$ for some $h\in \cC^2(M)$.  Then the operator $$L:=e^{-2h}\Delta_g$$ is associated with the Dirichlet energy
\[
\cE :  L^2(M,\bar{\nu}) \ni u \mapsto \int_M |d u|^2_g\di \nu_g \in [0,+\infty],
\]
in the sense that, for any $u \in L^2(M,\bar{\nu})$,
\[
\cE(u) = \int_M (Lu)u \di \bar{\nu}.
\]
The operator $L$ is also the weighted Laplacian of the weighted Riemannian manifold $(M^n,\bg, \bar\nu)$.  Indeed, for any $u\in \cC^\infty(M)$,
$$|du|^2_{\bg}=e^{-2h}|du|^2_g$$ thus
$$\int_M |du|^2_{\bg}\ \di \bar{\nu}=\cE(u).$$

The terminology ``time change'' comes from the fact that the Brownian motion on the  weighted Riemannian manifold $(M^n,\bg, \bar\nu)$ is obtained from the 
Brownian motion of $(M^n,g)$ only by a shift in time, see for instance \cite[Remark 8.3]{SturmJFA2018}. 

\subsection{Transformation rule}

The next lemma provides the transformation rule for the Bakry-\'Emery condition under time change. This rule is valid in a much more general setting, see for instance \cite{SturmJFA2018,SturmGAFA2020} and \cite{HanSturm}.
For completeness and because our  notation is slightly different, we provide a detailed proof.

\begin{lemm}\label{eq:BEHS} Let $(M^n,g)$ be a complete Riemannian manifold and $h\in \cC^2(M)$.  Set $\bg \df e^{2h}g$,  $\bar{\nu} \df e^{2h} \nu_g$ and  $L \df e^{-2h}\Delta_g$.  Then for any $q\in (0,+\infty]$ and $u \in \cC^\infty(M)$,
\begin{equation*}
\la dLu,du\ra_{\bg}-\frac12 L |du|^2_{\bg} \,  \ge \,  \frac{\left(L u\right)^2}{n+q}+\left( -\Ricm+\Delta_gh-c(n,q)|dh|^2_{g}\,\right) e^{-2h}|du|^2_{\bg}
\end{equation*}
where $c(n,q) = \frac{(n-2)(n+q-2)}{q} \, \cdot $
\end{lemm}

\begin{proof}[Proof of \lref{eq:BEHS}.] Recall the calculus rules:
\begin{equation}
\Delta_g(\varphi \phi) = \varphi \Delta_g \phi + \phi \Delta_g \varphi  - 2 \langle d \varphi, d \phi \rangle,
\end{equation}
\begin{equation}\label{eq:chain}
\Delta_g( \chi \circ \varphi) = (\chi' \circ \varphi) \Delta_g \varphi - (\chi'' \circ \varphi) |d \varphi|^2.
\end{equation}
Using them, we easily compute that 
$$\la dLu,du\ra_{\bg}=e^{-4h}\la d\Delta_gu,du\ra_g-2e^{-4h}\la dh,du\ra_g\,\Delta_gu \, ,$$
$$ L |du|^2_{\bg}=e^{-4h} \Delta_g |du|^2_g-2e^{-4h}|du|^2_g \Delta_g h-4e^{-4h}|du|^2_g|dh|^2_g+8e^{-4h}\nabla^g du(du,dh)\, , $$
so that
\begin{align*}
e^{4h}\left(\la dLu,du\ra_{\bg}-\frac12 L |du|^2_{\bg}\right) & =\la d\Delta_gu,du\ra_g-\frac12 \Delta_g |du|^2_g+\Delta_g h \,|du|^2_g \\
& +2|dh|^2_g|du|^2_g -4\nabla^g du(du,dh)-2\la dh,du\ra_g\Delta_gu.
\end{align*}
Let
$$A \df \nabla^g du+\frac{\Delta_g u}{n} g$$ 
be the traceless part of $\nabla^g du$. We introduce the tensor
$$B\df\left( dh\otimes du+du\otimes dh\right)$$
whose traceless part is
\[
\mathring{B} = B - \frac{\tr_g B}{n} g = B - 2\frac{\la du,dh\ra_g}{n} g.
\]
Using the Bochner formula \eqref{eq:Boch},  we get that
\begin{align*}
e^{4h}\left(\la dLu,du\ra_{\bg}-\frac12 L |du|^2_{\bg}\right)& =|A|_g^2-2\la A,\mathring{B}\ra_g\\
& + \Ric(du,du)+(\Delta_g h) |du|^2_g\\
& +2|dh|^2_g|du|^2_g+\frac{\left(\Delta_g u\right)^2}{n}\\
&+\frac{4}{n}\la du,dh\ra_g\Delta_gu-2\la dh,du\ra_g\Delta_gu.
\end{align*}
Then using that 
$$|A|_{g}^2-2\la A,\mathring{B}\ra_{g} \ge -|\mathring{B}|_{g}^2=-2\left( |dh|^2_g|du|^2_g+\la du,dh\ra_g^2\right)+\frac4n \la du,dh\ra_g^2,$$ 
we eventually obtain
\begin{align*}
e^{4h}\left(\la dLu,du\ra_{\bg}-\frac12 L |du|^2_{\bg}\right) & \ge\frac{\left(\Delta_g u\right)^2}{n}\\
&+\left(-\Ricm+\Delta_g h\right) |du|^2_g\\
&-\left(2-\frac4n\right)\la du,dh\ra_g^2\\
&-2\,\frac{n-2}{n}\la du,dh\ra_g\Delta_gu.
\end{align*}
By the Young inequality, we have
$$-2\,\frac{n-2}{n}\la du,dh\ra_g\Delta_gu\ge -\left(\frac1n-\frac{1}{n+q}\right)\left(\Delta_gu\right)^2-\frac{(n-2)^2n(n+q)}{n^2q}\la du,dh\ra_g^2,$$
hence the Cauchy-Schwarz inequality and some simple computations yield the desired inequality.
\end{proof}

\begin{coro}\label{cor:TC}Let $(M^n,g)$ be a Riemannian manifold and $\varphi\in \cC^2(M)$ such that $\varphi\ge 1$ and
\begin{equation}\label{eq:varphi}
\Delta_g \varphi-\lambda\Ricm\varphi\ge -\kappa \varphi
\end{equation}
for some  $\lambda>n-2$ and $\kappa\ge 0$.  If we set $h=\frac1\lambda \log\varphi$, then the weighted Riemannian manifold $(M^n,e^{2h}g, e^{2h} \nu_g)$ satisfies the $ \BE(\kappa/\lambda,n+q)$ condition, where  $q=(n-2)^2/(\lambda-(n-2)) $.
\end{coro}
\proof This is a direct consequence of \lref{eq:BEHS}.  Indeed,  note first that $\lambda = c(n,q)$ by our choice of $q$. Then the chain rule \eqref{eq:chain} implies that
$$\Delta_g h=\frac{\Delta_g \varphi}{ \lambda\varphi}+\frac{1}{\lambda}\frac{|d \varphi|^2_g}{\varphi^2}=\frac{\Delta_g \varphi}{ \lambda\varphi}+\lambda|dh|_g^2=\frac{\Delta_g \varphi}{ \lambda\varphi}+c(n,q)|dh|^2_{g}$$
so that,  using successively \eqref{eq:varphi} and $h\ge 0$, we get
\begin{align*}
\left( -\Ricm+\Delta_gh-c(n,q)|dh|^2_{g}\,\right) e^{-2h}|du|^2_{\bg} & = \left( \frac{\Delta_g \varphi}{ \lambda\varphi} -\Ricm\,\right) e^{-2h}|du|^2_{\bg}\\
&  \ge - \frac{\kappa}{\lambda} e^{-2h}|du|^2_{\bg} \ge - \frac{\kappa}{\lambda}  |du|^2_{\bg}. \qedhere
\end{align*}
\endproof

\section{Proof of Theorem \ref{Maintheo}}\label{pfA}

\subsection{Kato condition and the bottom of the spectrum} In this subsection, we recall a useful fact about Schr\"odinger operators whose potential satisfies a so-called Dynkin condition.  Let $(M^n,g)$ be a complete Riemannian manifold and $V\ge 0$ a locally integrable function on $M$.  For any $T >0$,  we define
$$\mbox{k}_T(V):=\sup_{x\in M} \iint_{[0,T]\times M} H(s,x,y)V(y)\di \nu_g(y)\di s.$$
It is classical (see e.g.~\cite{SV, Gueneysu-17}) that if $V$ satisfies the Dynkin condition
$$\Kato_T(V)<1,$$ then the quadratic form
$$ \cC_c^\infty(M) \ni u \mapsto \int_M\left( |du|_g^2-Vu^2\right)\di \nu_g$$ is bounded from below on $L^2(M,\nu_g)$,  hence it generates a self-adjoint operator $H_V=\Delta_g-V$ whose heat semi-group $\left\{e^{-tH_V}\right\}_{t>0}$ acts boundedly on each $L^p(M,\nu_g)$. More precisely, for any $p\in [1,+\infty]$ there exist $C>0$ and $\upomega\ge 0$ such that for any $t\ge 0,$
$$\left\| e^{-tH_V}\right\|_{L^p\to L^p}\le  C e^{\upomega t}.$$
 The proof of this classical result yields more precise information: 

\begin{prop}\label{prop:jaugeKato} Let $(M,g)$ be a complete Riemannian manifold. Let $V\ge 0$ be a locally integrable function on $M$ such that for some $T,\beta >0$,
$$\Kato_T(V)\le  1-e^{-\beta T}.$$ Then: 
\begin{enumerate}[i)]
\item for any $\phi\in \cC_c^\infty(M)$, 
\begin{equation}\label{eq:semibounded}
\int_M\left[\, |d\phi|_g^2-V\phi^2\, \right]\,\di \nu_g\ge -\beta  \int_M\phi^2\,\di \nu_g;
\end{equation}
\item $\text{spec } H_V\subset [-\beta,+\infty);$
\item for any $p\in [1,+\infty]$ and $t \ge 0$:
$$\left\| e^{-tH_V}\right\|_{L^p\to L^p}\le e^{\beta(t+T)}.$$
\end{enumerate}
\end{prop}
\proof Note that ii) follows from i) and the min-max characterization of the elements of $\text{spec } H_V$. Moreover, i) is a consequence of the case $p=2$ in iii).  Indeed, the latter implies that $e^{\beta(t+T)} \ge \lambda_1(e^{-tH_V}) = e^{-t\lambda_1(H_V)}$ for any $t>0$; taking the logarithm, dividing the resulting inequality by $t$, and letting $t$ tend to $+\infty$ gives $\lambda_1(H_V) \ge - \beta$ as desired. Therefore,  we need only to prove iii).

For any $\ell \in \N$, set $V_\ell \df \min( V,\ell)\un_{B_\ell(o)}$ and note that $\Kato_T(V_\ell)\le \Kato_T(V)$.  Since each $V_\ell$ is bounded, the quadratic form 
$$
\cC_c^\infty(M) \ni \phi \mapsto \int_M\left[\, |d\phi|_g^2-V_\ell\phi^2\, \right]\,\di \nu_g
$$
is bounded from below, hence the canonical Friedrichs extension  $H_{V_\ell}$ of the associated operator $\Delta_g - V_\ell$ is well-defined.  If $p\ge 0$ is such that
$$\left\| e^{-tH_{V_\ell}}\right\|_{L^p\to L^p}\le e^{\beta(t+T)}$$
for any $\ell$ and $t>0$, then the monotone convergence theorem ensures that the Friedrich extension $H_V$ of $\Delta_g - V$ and the pointwise limit $e^{-tH_{V}}$ of $\{ e^{-tH_{V_\ell}}\}$ are well-defined, and that $\left\| e^{-tH_{V}}\right\|_{L^p\to L^p}\le e^{\beta(t+T)}$. Therefore, from now on, we assume that $V$ is bounded with bounded support. 

Using selfadjointness and the Schur test, we need only to prove that for any $t \ge 0$,
$$M_V(t):= \left\|e^{-tH_{V}}\un\right\|_{L^\infty}\le e^{\beta(t+T)}.$$
Since we do not assume that $(M,g)$ is stochastically complete, the following Cauchy problem on $L^\infty(\mathbb{R}_+ \times M)$ may have more than one solution:
\begin{equation}\label{eq:Cauchy}\tag{*}
\left\{\begin{array}{l}
\left(\frac{\partial}{\partial t} +\Delta_g-V\right)u=0\\
u(0,\cdot)=\un.
\end{array}\right.
\end{equation}
On one hand, we define a solution $J$ as the monotone limit of the solutions $J_\ell$ of  the Cauchy problems:
$$\left\{\begin{array}{l}
\left(\frac{\partial}{\partial t} +\Delta_g-V\right)u=0\\
u(0,\cdot)=\un_{B_\ell(o)}.
\end{array}\right.$$
By the Duhamel formula applied to the Hilbert space $L^2(M)$ (see \cite[Chapter 4]{Lunardi}), we know that each $J_\ell : (t,x) \mapsto e^{-tH_V}\un_{B_\ell(o)}(x)$ satisfies that for any $t\in [0,T]$,
\begin{equation*}\label{eq:preDuhamel}
J_\ell(t,\cdot)=e^{-t\Delta_g}\un_{B_\ell(o)} +\int_0^t e^{-(t-s)\Delta_g}[V e^{-sH_V}\un_{B_\ell(o)}] \,  \di s.
\end{equation*}
Taking the monotone limit of both sides of the equation, we obtain
\begin{equation}\label{eq:Duhamel}
J(t,\cdot)=e^{-t\Delta_g}\un +\int_0^t e^{-(t-s)\Delta_g}[V e^{-sH_V}\un] \,  \di s.
\end{equation} 
Let us introduce the linear operator $\mcK\colon L^\infty([0,T]\times M)\rightarrow L^\infty([0,T]\times M)$ defined by
$$(\mcK u)(t,x)=\int_0^t\int_M H(t-s,x,y)V(y)u(s,y)\di\nu_g(y)\di s$$ for any $(t,x) \in [0,T]\times M$. Setting $f(t,x) \df e^{-t\Delta_g}\un(x)$, we can rewrite \eqref{eq:Duhamel} as
$$J=f+\mcK(J).$$ Notice that $\mcK$ preserves positivity, i.e.
$$u\ge 0\Rightarrow \mcK u\ge 0,$$ 
and that
$$\left\|\mcK \right\|_{L^\infty\to L^\infty }\le \Kato_T(V)\le  1-e^{-\beta T}.$$ From the latter, we get that $\mathrm{Id} - \mcK$ is invertible on $L^\infty([0,T]\times M)$ with inverse $\sum_{\ell \ge 0} \mcK^\ell$. Therefore, $J = (\mathrm{Id} - \mcK)^{-1}f$ and  $I\df  (\mathrm{Id} - \mcK)^{-1}\un$ satisfies
\begin{equation*}\label{eq:I}
I=\un+\mcK(I).
\end{equation*}
Using test functions, one easily checks that $I$ is a solution of \eqref{eq:Cauchy}. Since $f \le \un$ and $\mcK$ is positivity preserving, we have $\mcK^\ell f \le \mcK^\ell \un$ for any integer $\ell \ge 0$,  so that summing over $\ell$
yields
$$J\le I.$$ We easily have that 
\[
\un\le I\le e^{\beta T}
\]
so that for all $t\in [0,T]\colon$
$$M_V(t) = \|J(t,\cdot)\|_{L^\infty(M)}\le e^{\beta T}.$$
If $t>T$,  consider the integer $k$ such that $t\in [kT,(k+1)T]$. Using the semi-group law, one gets that
$$
M_V(t)\le M_V\left(k\,T\right)M_V\left(t-k\,T\right)\le  M_V\left(T\right)^k e^{\beta T} \le e^{\beta(k+1) T}\le e^{\beta(t+T) }.\eqno \qedhere$$
\endproof

The proof of the proposition easily yields that $\varphi$ in the next corollary is well-defined and satisfies the desired properties.

\begin{coro}\label{coro:jaugeKato}  Let $(M,g)$ be a complete Riemannian manifold and $V\ge 0$ a locally integrable function on $M$ such that  for some $T,\beta>0$,
$$\Kato_T(V)\le  1-e^{-\beta T}.$$  Then the equation
\begin{equation}\label{eq:equation}
H_V \varphi + 2 \beta \varphi = 2 \beta
\end{equation}
admits a weak solution $\varphi$ such that $1\le \varphi\le 2e^{\beta T}$ a.e.~on $M$.
\end{coro}

\begin{proof}
The argument to construct $I$ in the proof of the previous proposition allows to extend it to a function belonging to $L^\infty(\mathbb{R}_+ \times M)$ and satisfying
\begin{equation}\label{eq:control}
1 \le I(t,x)\le e^{\beta (T+t)}
\end{equation}
for any $(t,x) \in \mathbb{R}_+ \times M$.  Then we set
$$\varphi(x) \df 2\beta\int_0^{+\infty} e^{-2\beta t} I(t,x)\di t.\eqno \qedhere$$
\end{proof}

\begin{rem}\label{rk:reg} In the previous corollary, elliptic regularity implies that if $V$ is $\cC^{k,\alpha}$ for some $\alpha \in (0,1)$, then $\varphi$ is a strong $\cC^{k+2,\alpha}$ solution.
\end{rem}

\begin{rem}
\label{rem:dim2}
In dimension 2, we directly obtain the following.  For a complete Riemannian surface $(\Sigma^2, g)$, if there exists $T$ such that $\kato_T(\Sigma,g)$ is finite, then there exists a function $h \in \cC^2(M)$ such that 
$$0\leq h \leq 2 \kato_T(\Sigma,g) \log 4$$
and the conformal metric $g_h=e^{2h}g$ has Gauss curvature $K_{g_h}$ bounded from below by $-2\kato_T(\Sigma,g)\log(4)/T$. Indeed, it suffices to consider 
$$V=\frac{\Ricm}{2\kato_T(\Sigma,g)} $$
in the previous Corollary, $\varphi$ the corresponding weak solution of \eqref{eq:equation} and to define \mbox{$h=2\kato_T(\Sigma,g)\log(\varphi)$}. Then the result  follows from the transformation law of the Gauss curvature under conformal change
$$K_{g_h}=e^{-2h}(\Delta_g h+ K_g).$$
In this case, we do not need any restriction on the bound for the Kato constant, and moreover the $\RCD$ condition is satisfied with $N=2$. 
\end{rem}

\subsection{Proof of \tref{Maintheo}}

We are now in a position to prove \tref{Maintheo}.

\begin{proof}
Let $(M^n,g)$ be a complete Riemannian manifold satisfying \eqref{eq:NewDynkin}. Set
$$\lambda \df \frac12\left(n-2+\gamma^{-1}\right)$$ and consider $\beta>0$ such that
$$e^{-\beta T}=\frac12\left(1-(n-2)\gamma\right).$$
Since $\lambda > n-2$ and $\Ricm$ is a continuous function,  \cref{coro:jaugeKato} and Remark \ref{rk:reg} ensure that there exists $\varphi\in \cC^2(M)$  such that $1\le \varphi\le 2e^{\beta T}$ and \[\Delta_g \varphi-\lambda\Ricm  \varphi\ge -2\beta \varphi.\] 
Define
\[
h \df \frac{1}{\lambda} \log \varphi.
\]
Then \cref{cor:TC} implies that the weighted Riemanniann manifold $(M^n,e^{2h}g, e^{2h} \nu_g)$ satisfies the $ \BE(-2\beta/\lambda,n+q)$ condition with $\displaystyle q=\frac{2(n-2)^2\gamma}{1-(n-2)\gamma}\, \cdot$ Setting
\[
K  = K(n,\gamma) \df - \frac{4\ln(\frac{1}{2}(1-(n-2)\gamma))}{(n-2 + \gamma^{-1})}  \quad \text{and} \quad N = N(n,\gamma)  \df n + q,
\]
we get that $(M^n,e^{2h}g, e^{2h} \nu_g)$ satisfies the $ \BE(-K/T,N)$ condition.

Assume now that \eqref{eq:NewDynkin2} holds. We make a different choice for the  parameters $\beta$ and $\lambda$, namely
$$\beta\df 1/T \qquad \text{and} \qquad \lambda \df \frac{1-e^{-1}}{\katoT} \, ,$$
so that $q>0$ is given by
$$q=\frac{(n-2)^2\katoT}{1-e^{-1}-(n-2)\katoT} \, \cdot $$
Since $1-e^{-1}\ge \frac12$ and $1-e^{-1}-\frac13\ge \frac14$ we get that 
$$\frac1\lambda\le 2 \katoT \quad \text{and} \quad q\le 4(n-2)^2\katoT,$$
and then
$$0\le h\le \frac{1}{\lambda} \ln(2e)\le \frac{2}{\lambda}\le 4 \katoT.\eqno \qedhere $$\end{proof}

\section{Consequences on complete manifolds}\label{consman}

\subsection{Almost mononicity of the volume ratio} Theorem \ref{th:AlmostBishopGromov} is a direct consequence of the following proposition.

\begin{prop}\label{prop:towardsThC}Let $(M^n,g)$ be a complete Riemannian manifold satisfying \eqref{eq:NewDynkin2} for some $T>0$.  Then for any $x\in M$,  $\eta\in (0,1-1/\sqrt{2})$, $R\in (0, \sqrt{T}]$ and $r\in (0, (1-\eta) R]$,
\begin{equation}\label{eq:toprove}
\frac{\nu_g\left(B_R(x)\right)}{\nu_g\left(B_r(x)\right)}\le \left( \frac{R}{r} \right)^n  \exp\left(\frac{C(n)}{\log(1/(1-\eta))}\int_r^R \frac{{\mbox{k}_{s^2}(M^n,g)}}{s}\di s\right) \cdot
\end{equation}
\end{prop}

In order to prove this result,  we must recall some well-known facts. Consider $\kappa \ge 0$ and $N \in [1,+\infty)$.  As shown in \cite{Qian,Lott,BakryQian,WeiWylie},  the Bishop--Gromov comparison theorem holds on any complete weighted Riemannian manifold $(M,g,\bar\nu)$ satisfying the $ \BE(-\kappa,N)$ condition: for any $x\in M$ and $0< r<R$,$$
\frac{\bar\nu\left(B_R(x)\right)}{\bar\nu\left(B_r(x)\right)}\le \frac{\bV_{\kappa,N}(R)}{\bV_{\kappa,N}(r)}$$
where $$\bV_{\kappa,N}(\rho)\df\int_0^\rho \, \sinh^{N-1}(\sqrt{\kappa} s)\, \di s$$ for any $\rho>0$. We can classically bound the previous right-hand side from above to get the following estimate: for any $R\ge r >0$,
\begin{equation}\label{eq:BG}\frac{\bar\nu\left(B_R(x)\right)}{\bar\nu\left(B_r(x)\right)}\le e^{(N-1) \frac{\kappa R^2}{4} } \left(\frac{R}{r}\right)^N.\end{equation}
Indeed,  the inequality \[\frac{\di}{\di \sigma} \ln \left( \sinh(\sigma) \right) = \frac{\cosh(\sigma)}{\sinh(\sigma)}\le \frac1\sigma+\frac\sigma2\] holds for any $\sigma >0$.  For $\tau > \rho >0$, integrate the previous between $\rho$ and $\tau$ and apply the exponential function to the resulting inequality in order to get
\begin{equation}\label{eq:sinh}
\frac{\sinh(\tau)}{\sinh(\rho)} \le \frac{\tau}{\rho} e^{\frac{\tau^2-\rho^2}{4}} \, \cdot
\end{equation}
Then for any $0<r<R$,
\begin{align*}
\bV_{\kappa,N}(R)&=\int_0^R \, \sinh^{N-1}(\sqrt{\kappa} s)\, \di s\\
&=\frac Rr \ \int_0^r \, \sinh^{N-1}\left(\sqrt{\kappa}\frac{R}{r} s\right)\, \di s\\
&\le \frac Rr \ \int_0^r \, \left(\frac{R}{r}\right)^{N-1}e^{(N-1)\frac{\kappa R^2}{4}}\ \sinh^{N-1}\left(\sqrt{\kappa} s\right)\, \di s\\
&\le \left(\frac{R}{r}\right)^{N}e^{(N-1)\frac{\kappa R^2}{4}}\  \bV_{\kappa,N}(r).
\end{align*}

\proof[Proof of Proposition \ref{prop:towardsThC}]
Consider $x \in M$ and $\eta \in (0,1-1/\sqrt{2})$.  Define $\lambda(\tau)\df{\mbox{k}_{\tau^2}(M^n,g)}$ for any $\tau>0$. 

Let us first show that for any $0<r \le \sqrt{T}$ and $ \rho>r$,
\begin{equation}\label{eq:1}
\frac{\vol_g\left(B_{e^{-4\lambda(r)}\rho}(x)\right)}{\vol_g\left(B_r(x)\right)}\le  \left(\frac{\rho}{r}\right)^n\, \exp \left( C(n)\lambda(r)\left(\frac{\rho^2}{r^2}+\log\left(\frac{\rho}{r} \right)+1\right)\right).
\end{equation}
Since $\lambda$ is non-decreasing,  the assumption \eqref{eq:NewDynkin2} implies that for any $r \in (0,\sqrt{T})$,
\[
\lambda(r)\le \frac{1}{3(n-2)} \,\cdot 
\]
According to  \tref{Maintheo}, there exists $h\in \cC^2(M)$ with
\begin{equation}\label{eq:f}
0\le h\le 4\lambda(r)
\end{equation}
such that the weighted Riemannian manifold $(M^n,\, \overline{g}\df e^{2h}g,\, \bar \nu \df e^{2h}\nu_g)$ satisfies the $ \BE(-4\lambda(r)/r^2,n+4(n-2)^2\lambda(r) )$ condition.  Using an overline to denote the geodesic balls of the metric $\bg$,  for any $\rho> r$ inequality \eqref{eq:BG} leads to
\begin{equation*}
\frac{\bar \nu\left(\overline{B}_\rho(x)\right)}{\bar \nu\left(\overline{B}_r(x)\right)}\le  \left(\frac{\rho}{r}\right)^n \exp\left( 4(n-2)^{2}\lambda(r) \log\left(\frac{\rho}{r}\right)+(n+4(n-2)^2\lambda(r)-1)\lambda(r)\frac{\rho^2}{r^2}\right).
\end{equation*}
Using $\lambda(r)\leq 1/3(n-2)$ in the second summand of the exponential, we easily get
\begin{equation}\label{eq:2}\frac{\bar \nu\left(\overline{B}_\rho(x)\right)}{\bar \nu\left(\overline{B}_r(x)\right)}\le  \left(\frac{\rho}{r}\right)^n\, \exp\left( C(n)\lambda(r)\left(\frac{\rho^2}{r^2}+\log\left(\frac{\rho}{r}\right)\right)\right).\end{equation}
From \eqref{eq:f} we deduce that $$\overline{B}_r(x)\subset B_r(x)\text{, \quad }B_{e^{-4\lambda(r)}\rho}(x)\subset \overline{B}_\rho(x)$$
and
\begin{equation}\label{eq:meas}
\nu_g \le \bar\nu \le e^{8\lambda(r)} \nu_g,
\end{equation}
which easily lead to \eqref{eq:1} from \eqref{eq:2}.

We are now in a position to prove \eqref{eq:toprove} for $R \in (0,\sqrt{T}]$ and $r \in [R/2,(1-\eta)R]$.  Apply \eqref{eq:1} with $\rho = Re^{4 \lambda(r)}$:
\[
\frac{\vol_g\left(B_R(x)\right)}{\vol_g\left(B_r(x)\right)}\le  \left(\frac{R}{r}\right)^n\,  \exp\left( C(n)\lambda(r)\left(\frac{R^2}{r^2}+\log\left(\frac{R}{r}  \right) + 1 \right)\right).
\]
Since $r \ge R/2$, we deduce that
\[
\frac{\vol_g\left(B_R(x)\right)}{\vol_g\left(B_r(x)\right)}\le  \left(\frac{R}{r}\right)^n\,  \exp\left( C(n)\lambda(r)\right).
\]
Using  that $r \le (1-\eta)R$ and $\lambda$ is non-decreasing, we get
\[
\lambda(r) \le \frac{\lambda(r)}{\log(1/(1-\eta))} \int_r^R \frac{\di s}{s} \le \frac{1}{\log(1/(1-\eta))} \int_r^R \lambda(s) \frac{\di s}{s} \,
\]
so that \eqref{eq:toprove} is proved.

To conclude,  it remains to consider the case $r \in (0,R/2)$. Set $r_k \df (1-\eta)^{-k} r$ for any $k \in \mathbb{N}$. Let $\ell$ be the integer such that
\[
(1-\eta)^2R < r_\ell \le (1-\eta) R < \sqrt{T}. \]
Note that $r_\ell \in [R/2,(1-\eta)R]$ because $R/2<(1-\eta)^2R$.  Moreover, since
\[
(1-\eta) r_k = r_{k-1}  \ge (1-\eta)^2 r_k\ge r_k/2,
\]
we have $r_{k-1} \in [r_k/2, (1-\eta)r_k]$ for any $k \in \{1,\dots,\ell\}$. Therefore, the previous argument yields that
$$\frac{\vol_g\left(B_R(x)\right)}{\vol_g\left(B_{r_\ell}(x)\right)}\le \left( \frac{R}{{r_\ell}} \right)^{n}\ \exp \left( {\frac{C(n)}{\log(1/(1-\eta))}\int_{r_\ell}^R \lambda(s) \frac{\di s}{s}} \right) $$
and 
 $$\frac{\vol_g\left(B_{r_k}(x)\right)}{\vol_g\left(B_{r_{k-1}}(x)\right)}\le\left( \frac{{r_k}}{{r_{k-1}}} \right)^n\ \exp \left( \frac{C(n)}{\log(1/(1-\eta))}\int_{r_{k-1}}^{r_k} \lambda(s) \frac{\di s}{s} \right)$$
for any $k \in \{1,\dots,\ell\}$, and \eqref{eq:toprove} follows by taking the product of all these inequalities.
\endproof

\subsection{Existence of good cut-off functions}
\begin{prop}\label{prop:cutoff} Let $(M^n,g)$ be a complete Riemannian manifold satisfying \eqref{eq:NewDynkin2} for some $T>0$. Then for any $x\in M$ and $r\in (0,\sqrt{T})$, there exists $\chi_{x,r}\in \cC^4(M)$ such that:
\begin{enumerate}[i)]
\item $\chi_{x,r}=1$ on $B_{r/2}(x)$ and $\chi_{x,r}=0$ outside $B_{r}(x)$,
\item $\displaystyle \left|d\chi_{x,r}\right|_g^2+\left| \Delta_g \chi_{x,r}\right|\le \frac{C(n)}{r^2} \, \cdot $
\end{enumerate}
\end{prop}
\proof
By \tref{Maintheo}, there exists $h\in \cC^2(M)$ with
\begin{equation}\label{eq:contr}
0\le h\le 4\katoT\le 4/3
\end{equation}
such that $\left(M,\, \overline{g} \df e^{2h} g, \, \overline{\nu} \df e^{2h}\nu_g\right)$ is an $\RCD(-4/(3T),3n)$ space. Let $\overline{\dist}$ be the Riemannian distance associated with $\overline{g}$ and $\overline{B}_r(x)$ the $\overline{g}-$geodesic ball centered at $x \in M$ with radius $r>0$.  By \eqref{eq:contr}, we have
$$\dist_g\le \overline{\dist}\le e^{\frac43} \dist_g.$$ Since $e^{\frac 43} \le 4$ we get the inclusions
\begin{equation}\label{eq:inclusions}
B_{r/4}(x)\subset\overline{B}_r(x)\subset B_r(x)
\end{equation}
for any $x \in M$ and $r>0$.
According to \cite[Lemma 3.1]{MondinoNaber}, for any $z\in M$ and $ r\in(0,\sqrt{T})$ there exists $\phi_{z,r}\in \cC^4(M)$ such that:
\begin{itemize}
 \item $\phi_{z,r}=1$ on $\overline{B}_{r/2}(z)$ and $\phi_{z,r}=0$ outside $\overline{B}_{r}(z)$,
 \item $\displaystyle \left|d\phi_{z,r}\right|_{\overline{g}}^2+\left| L \phi_{z,r}\right|\le \frac{C(n)}{r^2} \, \cdot$
\end{itemize}
Now \eqref{eq:contr} yields
$$\left|d\phi_{z,r}\right|_{g}^2+\left| \Delta_g \phi_{z,r}\right|\le \frac{C(n)}{r^2}$$ and \eqref{eq:inclusions} implies that $\phi_{z,r}=1$ on $B_{r/8}(z)$ and $\phi_{z,r}=0$ outside of $B_r(z)$.

Now, consider $x\in M$ and $r\in (0,\sqrt{T})$. Let $\{z_i\}_{ i\in I} \subset  B_{r/2}(x)$ be such that the balls $\overline{B}_{r/(32)}(z_i)$ are disjoint one to another and $B_{r/2}(x) \subset \cup_i \overline{B}_{r/(16)}(z_i)$. The Bishop--Gromov comparison theorem on $\left(M,\overline{g}, \overline{\nu}\right)$ classically implies that there is an integer $N_1$ depending on $n$ only such that 
$\# I\le N_1.$ Set
$$\xi_{x,r} \df \sum_i \phi_{z_i,r/8}.$$ 
Then by construction $\xi_{x,r}\ge 1$ on $B_{r/2}(x)$. Moreover $\xi_{x,r}$ is zero outside
$$\cup_i \overline{B}_{r/8}(z_i)\subset \cup_i {B}_{r/2}(z_i)\subset B_r(x).$$
We easily get the estimate
$$\left|d\xi_{x,r}\right|_{g}^2+\left| \Delta_g \xi_{x,r}\right|\le \frac{C(n)N^2}{r^2} \, \cdot $$ 
Eventually, we set
\[
\chi_{x,r}\df u\circ \xi_{x,r}
\] where $u\in \cC^{\infty}(\R)$ is some fixed function such that $u=1$ on $[1,+\infty)$ and $u=0$ on $(-\infty,0]$.
\endproof
\begin{rem} The same proof also shows that  if $$\katoT\le \gamma<\frac{1}{n-2}$$ then for any $x\in M$ and $r\in (0,\sqrt{T})$, there exists $\chi_{x,r}\in \cC^4(M)$ such that 
\begin{enumerate}[i)]
\item $\chi_{x,r}=1$ on $B_{r/2}(x)$ and $\chi_{x,r}=0$ outside $B_{r}(x)$,
\item $\displaystyle  \left|d\chi_{x,r}\right|_g^2+\left| \Delta_g \chi_{x,r}\right|\le \frac{C(n,\gamma)}{r^2} \, \cdot$
\end{enumerate}
 
\end{rem}

As mentioned in the introduction and in \cite[Remark 3.4]{CMT1}, the existence of cut-off functions like in the previous proposition implies that all the results of \cite{CMT1,CMT2} extend to complete Riemannian manifolds.  Indeed, our previous work relies, among others, on a Li--Yau type inequality: the restriction to the case of closed Riemannian manifolds was then due to the fact that this inequality \cite[Proposition 3.3]{Carron:2016aa} was proved only for closed  manifolds. Moreover,  it is known that a complete Riemannian manifold  $(M^n,g)$ with cut-off functions as above and such that $\katoT\le 1/(16n)$ satisfies the same Li--Yau type inequality (see \cite[Proposition 3.16]{Carron:2016aa}), this allowing to apply our results in the complete setting. We will not state all the results of \cite{CMT1,CMT2} that now hold true on complete Riemannian manifolds but we will  focus on some key results.

\subsubsection{Monotonicity of heat ratios}  Let $(M^n,g)$ be a complete Riemannian manifold. For any $t>0$ and $x,y\in M$,  set \[U(t,x,y)\df-4t\log\left( (4\pi t)^{\frac n2} H(t,x,y)\right).\]

\begin{theo}\label{Theo:heatratio} Let $(M^n,g)$ be a complete Riemannian manifold such that for some $T>0$,
$$\katoT\le \frac{1}{16n} \quad \text{and}\quad \int_0^T \frac{{\mbox{k}_{t}(M^n,g)}}{t}\di t<\infty.$$
For any $t\in (0,T)$, set
$$\Phi(t)\df \int_0^t \frac{{\mbox{k}_{\tau}(M^n,g)}}{\tau} \di \tau.$$
Then for any $t\in (0,T)$ and $s>0$ there exists $\overline{\lambda}=\overline{\lambda}(n,\Phi(T), s/t)>0$ such that 
$\lim_{\sigma\to 0+} \overline{\lambda}(n,\Phi(T), \sigma)=0$ and the function 
$$\lambda \in (0,\overline{\lambda}]\mapsto e^{c_n\Phi(\lambda t)\left(\frac ts -\frac st\right)}\int_M \frac{e^{-\frac{U(\lambda t,x,y)}{4\lambda s}}}{\left(4\pi \lambda s\right)^{\frac n 2}} \di \nu_g(y)$$ is monotone. It is non-increasing if $s\ge t$ and non-decreasing if $s\le t.$
\end{theo}
\begin{rem}
\label{rem:newstrong} In \cite[Corollary 5.10]{CMT1} we used the Li--Yau type inequality to prove the previous monotonicity under the assumptions 
$$\katoT\le \frac{1}{16n} \quad \text{and}\quad \int_0^T \frac{{\sqrt{\mbox{k}_{t}(M^n,g)}}}{t}\di t<\infty.$$
A close look at the proof of  \cite[Proposition 3.3]{Carron:2016aa} shows that a different choice of the parameters $\delta,\alpha$, namely
$$\delta\simeq \left(\mbox{k}_{T}(M^n,g)\right)^2\text{ and } \alpha=1-\sqrt{\frac{n\delta}{2-\delta}} \, ,$$ leads to a Li--Yau inequality where the term $\sqrt{\mbox{k}_{T}(M^n,g)}$ is replaced by a multiple of $\mbox{k}_{T}(M^n,g)$. Theorem \ref{Theo:heatratio} is then obtained by using this latter version of the Li--Yau inequality in the proof of \cite[Corollary 5.10]{CMT1}.
\end{rem}
\subsubsection{Local doubling and Poincaré}  As \cite{Carron:2016aa} shows, the validity of the Li--Yau inequality on a complete Riemannian manifold  satisfying 
$$\katoT\le \frac{1}{16n}$$ 
implies that $(M^n,g)$ is locally doubling and satisfies the local $L^2$ Poincaré inequality; note that the latter implies the $L^{2-\eps}$ one for some $\eps>0$, see \cite{KeithZhong}.  Below, we prove the local doubling property and the  local $L^1$ Poincaré inequality, using the fact that both properties are preserved under a bi-Lipschitz change of the metric and the measure.
\begin{prop}\label{prop:DP} Let $(M^n,g)$ be a complete Riemannian manifold satisfying \eqref{eq:NewDynkin}. Then there exist $C,N,\lambda$ depending only on $n$ and $\gamma$ such that  for any $x\in M$ and $0<r\le \sqrt{T}$,
\begin{enumerate}
\item for any $s\in (0,r)$,
\[
 \nu_g\left(B_r(x)\right)\le C\left(\frac{r}{s}\right)^N \,\nu_g\left(B_s(x)\right);
\]
\item for any $\varphi\in \cC^1\left(B_r(x)\right)$ with $\displaystyle \int_{B} \varphi \di \nu=0$,
$$(\star)\ \ \left \|\varphi\right\|_{L^1\left(B_r(x)\right)}\le\lambda r \,  \left\|d\varphi\right\|_{L^1\left(B_r(x)\right)}.$$
\end{enumerate}
\end{prop}
\proof According to \tref{Maintheo} there exist $K\ge 0$ and $N>n$ depending both on $n$ and $\gamma$ only, and $h\in \cC^2(M)$ with
$$1\le e^h\le C(n,\gamma),$$ such that the weighted Riemannian manifold $\left(M^n,\bg \df e^{2h} g, \bar \nu \df e^{2h}\nu_g\right)$ satisfies the $\BE\left(-K/T, N\right)$ condition.  The Bishop-Gromov inequality \eqref{eq:BG} implies that the $\bar \nu$-measure of the $\bg$-geodesic ball satisfies :
$$\forall x\in M,0< r<R\colon \bar \nu\left(\overline{B}_R(x)\right)\le e^{C(n,\gamma) \frac{R^2}{T^2}} \left(\frac{R}{r}\right)^N\, \bar\nu\left(\overline{B}_r(x)\right).$$
Using this estimate, \eqref{eq:meas} and \eqref{eq:inclusions}, we get that  for $0<s\le r\le \sqrt{T}$ and $x\in M$,
\begin{align*}
\nu_g\left(B_r(x)\right)&\le \bar \nu \left(B_r(x)\right)  \\
& \le  \bar \nu \left(\overline{B}_{4r}(x)\right)
\\
&\le C(n,\gamma)e^{C(n,\gamma) \frac{r^2}{T}} \left(\frac{r}{s}\right)^N\, \bar \nu \left(\overline{B}_{s}(x)\right)\\
&\le C(n,\gamma) \left(\frac{r}{s}\right)^N\, \nu_g \left(\overline{B}_{s}(x)\right)\\
&\le C(n,\gamma) \left(\frac{r}{s}\right)^N\, \nu_g \left({B}_{s}(x)\right).
\end{align*}
According to \cite{Sturm2006II,LottVillani2007,vonRenesse}, see also \cite[Corollary 19.13]{Villani_2009}, we also have the $L^1$-Poincaré inequality : if $x\in M$ and $r>0$ then
for any $\varphi\in \cC^1\left(\overline{B}_r(x)\right)$ with $\displaystyle \int_{\overline{B}_r(x)} \varphi \di \bar \nu=0$,
$$\int_{\overline{B}_r(x)} |\varphi| \di \bar \nu\le C(n,\gamma) e^{C(n,\gamma) \frac{r^2}{T}}\, r\,  \int_{\overline{B}_r(x)} |d\varphi|_{\bg} \di \bar \nu.$$
Then if $\varphi\in \cC^1\left(\overline{B}_{4r}(x)\right)$ with $c= \int_{\overline{B}_r(x)} \varphi \di \bar \nu$ one gets
\begin{align*}\left \|\varphi-c\right\|_{L^1(B_r(x))}&\le  \int_{\overline{B}_{4r}(x)}\left| \varphi -c\right| \di \bar \nu\\
&\le C(n,\gamma) e^{C(n,\gamma) \frac{r^2}{T}}\, r\,  \int_{\overline{B}_{4r}(x)} |d\varphi|_{\bg} \di \bar \nu\\
&\le C(n,\gamma)^2e^{C(n,\gamma) \frac{r^2}{T}}\, r\,  \int_{{B}_{4r}(x)} |d\varphi|_{ g} \di  \nu.
\end{align*}
Moreover we always have
$$\left \|\varphi-\varphi_{B_r(x)}\right\|_{L^1\left(B_r(x)\right)}\le 2\left \|\varphi-c\right\|_{L^1\left(B_r(x)\right)}.$$
In order to conclude that we get the stronger Poincaré inequality $(\star)$, we refer to the work of Jerison and of  Maheux, Saloff-Coste \cite{Je, MSc}.
\endproof

\section{Consequences for limit spaces}\label{conslim} In this section, we explain how the previous results broaden the study carried out in \cite{CMT1,CMT2} on limits of Riemannian manifolds with suitable uniform bounds on the Ricci curvature. We begin with a convenient definition.

\begin{defi}\label{def:dynkin} We say that a pointed metric measure space $(X,\dist,\mu,o)$ is a renormalized limit space if it is the pointed measured Gromov--Hausdorff limit of a sequence of pointed complete weighted Riemannian manifolds of same dimension $\left\{(M_\ell,g_\ell,\mu_\ell \df c_\ell \nu_{g_\ell},o_\ell)\right\}$ with $\{c_\ell\} \subset (0,+\infty)$ such that there exists $\kappa>0$ satisfying that, for any $\ell$,
\begin{equation}
 \kappa^{-1} c_\ell\le \nu_{g_\ell}\left(B_{\sqrt{T}}(o_\ell)\right)\le \kappa c_\ell.
\end{equation}
\end{defi}

\begin{rems} 
\begin{enumerate}
\item[]
\item We may denote a renormalized limit space by $(X,\dist,\mu,o) \leftarrow (M_\ell^n,g_\ell,\mu_\ell,o_\ell)$ if needed.

\item A common renormalization is $c_\ell=\nu_{g_\ell}\left(B_{\sqrt{T}}(o_\ell)\right)^{-1}$ for all $\ell$.
\end{enumerate}
\end{rems}

\subsection{Dynkin limit spaces}

Consider a sequence $\left\{(M_\ell,g_\ell,\mu_\ell ,o_\ell)\right\}$ as in Definition \ref{def:dynkin} and assume that it additionnally satisfies \eqref{eq:uniform} for uniform $T>0$ and $\gamma>0$. In \cite{Carron:2016aa,CMT1}, it was shown that if the manifolds are closed and $\gamma < 1/(16n)$, then the sequence $\left\{(M_\ell,g_\ell,\mu_\ell,o_\ell)\right\}$ is uniformly doubling, so that it admits limit points in the pointed measured Gromov--Hausdorff topology thanks to Gromov's compactness theorem.  The  constant $1/(16n)$ was used in a fixed point argument leading to the uniform doubling condition.  We called such limit spaces Dynkin limits. 

In the present paper,  the doubling condition given by \pref{prop:DP} allows to work in a more general setting where closedness is removed and $\gamma < 1/(n-2)$. This is why we adopt the following new definition.

\begin{defi}\label{def:renormalized} A renormalized limit space $(X,\dist,\mu,o) \leftarrow (M_\ell^n,g_\ell,\mu_\ell,o_\ell)$ is called a Dynkin limit space if there exist $T>0$ and $\gamma \in(0, 1/(n-2))$ such that \eqref{eq:uniform} holds.
\end{defi}

We prove the following result for Dynkin limit spaces.  It extends and refines \cite[Theorem 1.1]{CMT2}.

\begin{prop}\label{prop:Dynlinreg} Any Dynkin limit space $(X,\dist,\mu,o)$ is $k$-rectifiable for some integer $k \in \{2,\ldots,n\}$,  where $n$ is the dimension of the approximating manifolds. This means that there exists a countable collection $\{(V_i,\phi_i)\}_i$ such that $\{V_i\}$ are Borel subsets covering $X$ up to a $\mu$-negligible set and $\phi_i : V_i \to \R^{k}$ is a bi-Lipschitz map satisfying $(\phi_i)_\#(\mu \measrestr V_i) \ll \cH^{k}$ for any $i$. 
\end{prop}
\proof Let  $(X,\dist,\mu,o) \leftarrow (M_\ell^n,g_\ell,\mu_\ell = c_\ell \nu_{g_\ell},o_\ell)$ be a Dynkin limit space. For any $\ell$, \tref{Maintheo} provides $K \ge 0$ and $N \in [1,+\infty)$ depending on $n$ and $\gamma$ only and a function $h_\ell\in \cC^2(M_\ell)$ with $0\le h_\ell\le C(n,\gamma)$ such that the weighted Riemannian manifold $(M_\ell^n,  \bg_\ell\df e^{2h_\ell} g_\ell, \bar\nu_\ell\df e^{2h_\ell} \nu_{g_\ell})$ satisfies the $\mathrm{RCD}(-K/T,N)$ condition. The space of pointed  $\RCD(-K/T, N)$ spaces is compact in the pointed measured Gromov--Hausdorff topology, hence we can assume that, up to extracting a subsequence, the sequence $\{ (M^n_\ell, \dist_{\bg_\ell},\bar \mu_\ell \df c_\ell \bar\nu_\ell,o_\ell)\}$ converges to some $\RCD(-K/T, N)$ space $(X,\bar \dist,\bar \mu,o)$. By \cite[Theorem 1.1]{MondinoNaber} and \cite[Theorem 0.1]{Bru__2020} there exists $k\in \{0,\ldots, \lfloor N \rfloor \}$ such that $(X,\bar \dist,\bar \mu)$ is $k$-rectifiable. But $\dist\le \bar \dist\le e^{C(n,\gamma)} \dist$ and $\mu\le \bar \mu\le e^{2C(n,\gamma)} \mu,$ hence $(X, \dist, \mu)$ is $k$-rectifiable too. That $k$ is at most $n$ follows from the lower semicontinuity of the essential dimension of finite-dimensional $\RCD$ spaces under pointed measured Gromov--Hausdorff convergence \cite{Kitabeppu}, see also \cite{BPS}.
\endproof
\begin{rem}{\bf About the Mosco convergence of the Energy forms.}
In the setting of the proof of \pref{prop:Dynlinreg},  the Dirichlet forms $\cE_\ell$ and $\bar \cE_\ell$ defined for any $\ell$ by
$$\cE_\ell(u)\df \int_{M_\ell} |du|^2_{g_\ell} d\mu_\ell, \qquad \bar \cE_\ell(u)\df \int_{M_\ell} |du|^2_{\bg_\ell} d\bar \mu_\ell,$$
coincide. We know that the pointed measured Gromov-Hausdorff convergence of $\RCD$ spaces implies the Mosco convergence of the Cheeger energy \cite{GMS}, hence we get that the sequence $\left\{(M_\ell,\dist_{g_\ell}, \mu_\ell,\cE_\ell,  o_\ell)\right\}_\ell$ converges in the pointed Mosco--Gromov--Hausdorff topology to 
$(X,\dist,\mu,\cE,o)$ where $\cE$ is the Cheeger energy of $(X,\bar \dist,\bar \mu)$.  This does not implies, a priori,  that in the setting of  \dref{def:dynkin}, the pointed measured Gromov--Hausdorff convergence self-improves to a Mosco convergence of the energy. Indeed, several choices of functions $f_\ell$ can be made and the limit distance $\bar \dist$ and limit measure $\bar \mu$ could depend on the subsequence. Compare with \cite[Theorem 4.8]{CMT1}.
\end{rem}
\subsection{Kato limit spaces}

Let $T>0$ and $f\colon (0,T]\rightarrow \R_+$ be a non-decreasing function satisfying $$f(T)<\frac1{n-2}\quad \text{and} \quad \lim_{t\to 0+} f(t)=0.$$
The next definition is the natural variant of the one of Kato limit used in our previous articles (recall the discussion at the beginning of the previous section).

\begin{defi}\label{def:kato} A renormalized limit space $(X,\dist,\mu,o) \leftarrow (M_\ell^n,g_\ell,\mu_\ell,o_\ell)$ is called a Kato limit space associated to $f$ if $\mbox{k}_{t}(M_\ell^n,g_\ell)\le f(t)$ for any $\ell$ and $t \in(0,T]$.
\end{defi}
\pref{prop:Dynlinreg} together with \cite[Theorem 4.4]{CMT2} implies the following:
\begin{prop}\label{prop:Katoreg} Let $(X,\dist,\mu,o) \leftarrow (M_\ell^n,g_\ell,\mu_\ell,o_\ell)$ be a  Kato limit space. Then there exists $k\in \{0,\ldots, n\}$ such that $(X,\dist,\mu)$ is $k$-rectifiable, and for $\mu$-almost every $x\in X$ the space $(\R^k, \dist_{\text{eucl}}, \cH^k,0)$ is the unique metric measure tangent cone of $(X,\dist,\mu)$  at $x$. \end{prop} 

\subsection{Non-collapsed strong Kato limit spaces} For $T>0$, let $f\colon (0,T]\rightarrow \R_+$ be a non-decreasing function satisfying \eqref{eq:newstrong}.  In our previous work, we defined strong Kato limit spaces through the stronger integrability condition
\[
\int_0^T \frac{\sqrt{f(s)}}{s} \di s <\infty.
\]
However, as explained in Remark \ref{rem:newstrong},  we are now in a position to work with the weaker \eqref{eq:newstrong}. This leads to the following new definition.

\begin{defi}\label{def:ncSkato} A pointed metric space $(X,\dist,o)$ is called a non-collapsed strong Kato limit space associated to $f$ if there exists a sequence of pointed complete Riemannian manifolds $\{ (M_\ell^n,g_\ell,o_\ell)\}$ such that :
\begin{enumerate}
\item $\mbox{k}_{t}(M_\ell^n,g_\ell)\le f(t)$ for any $\ell$ and $t \in(0,T]$, 
\item there exists $v >0$ such that $\nu_{g_\ell}(B_{\sqrt{T}}(o_\ell)) \ge v$ for all $\ell$,
\item $(M_\ell^n,g_\ell,o_\ell) \to (X,\dist,o)$ in the pointed Gromov--Hausdorff topology.
\end{enumerate}
\end{defi}

\tref{th:AlmostBishopGromov} has the following consequence.

\begin{prop}
Let $(X,\dist,o) \leftarrow (M_\ell^n,g_\ell,o_\ell)$ be a non-collapsed strong Kato limit space. Then for any $x \in X$,  the volume density 
$$\uptheta_X(x)=\lim_{r\downarrow 0} \frac{\cH^n(B_{r}(x))}{\omega_n r^n}$$ is well-defined. Moreover,  there exists a function $c : (0,+\infty) \to (0,+\infty)$ depending on $n$,  $T$ and $f$ only such that for any $x \in X$,
\begin{equation}\label{eq:densitycontrol}
c(\dist(o,x)) v\le \uptheta_X(x)\le 1.
\end{equation}
\end{prop}

\begin{proof}
Because of the existence of cut-off functions established in Proposition \ref{prop:cutoff}, the adaptation of the proof of Colding's volume convergence theorem \cite{Colding_1997,  CheegerPisa} done in \cite[Section 7, especially Proposition 7.5]{CMT1} carries over to the complete setting. As a consequence,  one has $(M_\ell, g_\ell,\nu_{g_\ell}, o_\ell) \to (X,\dist,\cH^n,o)$ in the pointed measured Gromov-Hausdorff topology.  Therefore, the conclusion of  \tref{th:AlmostBishopGromov} passes to the limit and yields that, for any $x \in X$, $R \in (0, \sqrt{T}]$,  $\eta \in (0,1-1/\sqrt{2})$ and $r \le (1-\eta)R$,
\[
\frac{\cH^n(B_{R}(x))}{\omega_n R^n}  \exp \left(-\frac{C(n)\Phi(R)}{\eta} \right) \le \frac{\cH^n(B_{r}(x))}{\omega_n r^n} \exp\left(-\frac{C(n)\Phi(r)}{\eta}\right).
\]
Taking successively the limit inferior as $r \downarrow 0$ and then the limit superior as $R \downarrow 0$,  we obtain that $\theta_X(x)$ is well-defined.   As for  \eqref{eq:densitycontrol}, the lower bound is obtained as in \cite[Remark 2.18]{CMT1} and the upper bound  as in \cite[Corollary 5.13]{CMT1}.
\end{proof}

In \cite[Corollary 5.20]{CMT2}, we proved that there exists $\delta=\delta(n,f)\in (0,1)$ such that for any non-collapsed strong Kato limit space $X$ associated to $f$,  the dense open subset
$$\{x\in X\colon \theta_X(x)>1-\delta\}$$
is a topological manifold with H\"older regularity. This result was a consequence of the intrinsic Reifenberg theorem of Cheeger and Colding \cite[Theorem A.1.1]{ChCo97} applied to balls $B_{\sqrt{\tau}}(x)$ where the heat ratio is almost $1$,  i.e.~such that for small enough $\delta>0$ ,
$$1 \le (4\pi \tau)^{\frac n2}\, H(\tau,x,x) \le \frac{1}{1-\delta}\cdot $$
Indeed,  as pointed out in \cite[Remark 5.4]{CMT2}, we have
\[
\theta_X(x)^{-1} = \lim_{t \to 0} (4\pi \tau)^{\frac n2}\, H(\tau,x,x)
\]
In order to apply the Cheeger--Colding--Reifenberg theorem, we had to prove a Reifenberg property for these balls $B_{\sqrt{\tau}}(x)$. Recall that the latter asks that for any $y\in B_{\sqrt{\tau}}(x)$ and $s\in (0,\sqrt{\tau}/2)$,
\begin{equation}\label{eq:Reif}
\dist_{GH}\left( B_s(y), \bB_s^n\right)\le \epsilon \,s.
\end{equation}
We obtained these property through two key results:
\begin{itemize}
\item the almost monotonicity of the heat ratio (that is to say  Theorem \ref{Theo:heatratio} for the special value $s=t/2=\tau/4$);
\item a rigidity result for $\RCD(0,n)$ spaces $(Z,d_Z,\cH^n)$ for which there exist $\tau>0$ and $z\in Z$ such that $(4\pi \tau)^{\frac n2}\, H(\tau,z,z)=1$ (see \cite[End of Proof of Theorem 5.9]{CMT2}).
\end{itemize}
Similarly, the almost monotonicity of the volume ratio granted by \tref{th:AlmostBishopGromov} and the rigidity of $\RCD(0,n)$ spaces with maximal volume ratio \cite[Theorem 1.1]{DPG} naturally lead to the following result for balls with almost maximal volume:
\begin{theo}\label{th:reifenberg}Let $(X,\dist,o)$ be a non-collapsed strong Kato limit space associated to $f$. Then for any $\epsilon\in (0,1)$ there exists $\delta\in (0,1)$ depending only on $n,f,\epsilon$ such that if $x\in X$ and $r \in (0, \delta\sqrt{T}]$ satisfy $$\frac{\cH^n\left(B_r(x)\right)}{\omega_n r^n}\ge 1-\delta$$ then
the ball $B_{r/2}(x)$ satisfies the Reifenberg property \eqref{eq:Reif}.
\end{theo}

As in \cite{CMT2}, there are two alternative ways for deriving the H\"older regularity of balls satisfying the Reifenberg property: 
either by the intrinsic Reifenberg theorem of Cheeger and Colding mentioned above, or as a consequence of a recent idea of Cheeger, Jiang and Naber \cite{CJN} based on a transformation theorem. The advantage of the latter approach is that it gives a more quantitative statement and proves that harmonic almost splitting maps are bi-Hölder.  In \cite[Theorem 5.14]{CMT2} we proved a transformation theorem under a strong Kato bound,  which can be easily rephrased for complete manifolds and under our weaker condition \eqref{eq:newstrong} on the function $f$. Then the same argument as in \cite{CJN} yields the following :

 \begin{theo}\label{theo:CJN} Let $f\colon (0,T]\rightarrow \R_+$ be a non-decreasing function satisfying $$f(T)<\frac1{n-2}\text{ and }\int_0^T\frac{ f(t)}{t} \di t<\infty.$$ For any $\alpha\in (0,1)$ there exists $\delta\in (0,1)$ depending only on $n,f,\alpha$ such that if $(M^n,g)$ is a complete Riemannian manifold satisfying \eqref{eq:bound}, for any   $x\in M$ and $r\in (0,\sqrt{T})$ for which
 \begin{equation}\label{controlKato}\mbox{k}_{r^2}(M^n,g)<\delta\end{equation}
 and there exists a harmonic map $h\colon B_r(x)\rightarrow \R^n$ with $h(x)=0$ and
\begin{equation}\label{controlh}r^2\fint_{B_r(x)} |\nabla^g dh|^2 \di \nu_g+\fint_{B_r(x)} \left|dh{}^tdh-\text{I}_n\right| \di \nu_g \le \delta,\end{equation} then
 \begin{enumerate}[i)]
 \item $h\colon B_{r/2}(x)\rightarrow \R^n$ is a diffeomorphism onto its image;
 \item $\bB^n_{\alpha r/2}\subset h\left(B_{r/2}(x)\right)$;
 \item $\displaystyle \forall y,z\in B_{r/2}(x)\colon
\alpha r^{1-\frac1\alpha}\,\dist_g^{\frac1 \alpha}(y,z) \le  \| h(y)-h(z)\|\le \alpha^{-1}\, \dist_g(y,z).$
 \end{enumerate}
 \end{theo}
\begin{rems} We conclude with some remarks on the previous statement. 
\begin{itemize}
\item Following \cite[Theorem 1.2]{ChCo00} or \cite[Theorem A.1]{CMT3},  it can be shown that the existence of a harmonic map satisfying $\eqref{controlh}$ implies that 
$$\displaystyle \nu_g\left(B_r(x)\right)\ge \left(1-C(n)\sqrt{\delta}\right)\omega_n r^n.$$
\item The condition  $\mbox{k}_{r^2}(M^n,g)<\delta$ is satisfied for $r$ small enough, that is if $$ r\le \sqrt{T} \exp\left(-\int_0^T\frac{f(s)}{\delta s} \di s\right).$$
\item It can be shown that if $f$ is as \tref{theo:CJN} , then for any $\delta\in (0,1)$ there exists $\eta\in (0,1)$ depending only on $n,f,\delta$ such that if $(M^n,g)$ is a complete Riemannian manifold such that for all $t\in (0,T]$ and for some $r\in (0, \sqrt{T}]$
$$\mbox{k}_{t}(M^n,g)\le f(t)\text{ , } \mbox{k}_{r^2}(M^n,g)<\eta\text{ and 
 }\frac{\cH^n\left(B_{2r}(x)\right)}{\omega_n (2r)^n}\ge 1-\eta$$
then there exists a harmonic map $h\colon B_r(x)\rightarrow \R^n$ satisfying \eqref{controlh} (compare with \cite[Corollary 5.13]{CMT2}).
\end{itemize}
\end{rems}

\hfill

\subsection*{Statements and Declarations}

\subsubsection*{Conflict of interest. }The authors have no relevant financial or non-financial interests to disclose.

\subsubsection*{Data availability.} Data sharing not applicable to this article as no datasets were generated or analysed during the current study.

\bibliographystyle{alpha} 
\bibliography{Ref.bib}
\end{document}